\documentclass[11pt]{article}
\usepackage[margin=1in]{geometry}
\usepackage{amsmath,amssymb,amsthm,mathtools}
\usepackage{mathrsfs,mathabx}  
\usepackage{microtype}
\usepackage{enumitem,tikz}
\usepackage[colorlinks=true,linkcolor=blue,citecolor=blue,urlcolor=blue]{hyperref}

\newtheorem{theorem}{Theorem}[section]
\newtheorem{lemma}[theorem]{Lemma}
\newtheorem{proposition}[theorem]{Proposition}
\newtheorem{corollary}[theorem]{Corollary}

\newtheorem{question}[theorem]{Question}
\newtheorem{observation}[theorem]{Observation}
\newtheorem{remark}[theorem]{Remark}
\newtheorem{example}[theorem]{Example}

\usepackage[comma,sort, numbers]{natbib}

\allowdisplaybreaks

\newcommand{\R}{\mathbb R}
\newcommand{\V}{\mathcal V}  
\newcommand{\join}{\vee}

\title{ Large Planar Point Sets Contain 4 Collinear Points or Almost 7-Cliques, and Related Results }

\author{Bhaswar B. Bhattacharya\thanks{Department of Statistics and Data Science, University of Pennsylvania and National University of Singapore \texttt{bhaswar@wharton.upenn.edu}} \and 
Sandip Das\thanks{Advanced Computing and Microelectronics Unit, Indian Statistical Institute, Kolkata, \texttt{sandip.das.69@gmail.com}}
\and 
Sk Samim Islam\thanks{Advanced Computing and Microelectronics Unit, Indian Statistical Institute, Kolkata, \texttt{samimislam08@gmail.com}}
\and 
Aashirwad Mohapatra\thanks{Advanced Computing and Microelectronics Unit, Indian Statistical Institute, Kolkata, \texttt{aashirwad\_r@isical.ac.in}} 
\and 
Saumya Sen\thanks{Advanced Computing and Microelectronics Unit, Indian Statistical Institute, Kolkata, \texttt{saumyasen72@gmail.com}}
}

\date{  }

\begin{document}
\maketitle

\begin{abstract}
We prove that every sufficiently large finite planar point set contains either four collinear points or seven points with at most one non-visible pair. More generally, we show that for every fixed graph $H$ with chromatic number at most five, or with chromatic number six and a color-critical edge, the visibility graph of every sufficiently large finite planar point set with no four collinear points contains a copy of $H$. These results extend the recent breakthrough of \citet{Bonnet}, guaranteeing six pairwise visible points, and come within one visibility edge of the next open case of the big-line-big-clique conjecture.  
\end{abstract}

\section{Introduction}

Given a finite point set $P\subset\R^2$, two distinct points $p,q\in P$ are said to be \emph{visible} in $P$ if the open line segment joining them contains no point of $P$. The \emph{visibility graph} of $P$, to be denoted by $\V_P$, has vertex set $P$, with two distinct vertices adjacent if and only if they are visible in $P$. The celebrated \emph{big-line-big-clique conjecture} of \citet{KaraPorWood} asserts that, for every pair of integers $k,\ell\ge3$, every sufficiently large finite planar point set contains either $\ell$ collinear points or $k$ pairwise visible points (see \cite{LeuchtnerNicolasSuk,PorWoodInfinite,Pfender,PorWood,Matousek}
and references therein for other related results and conjectures).

If no three points of a set are collinear, then the visibility graph is complete, hence,  the big-line-big-clique conjecture holds trivially for $\ell=3$ and every $k \geq 3$. \citet{KaraPorWood} established the conjecture for $k\le4$ and all $\ell \geq 4$. The case $(k,\ell)=(5,4)$ was subsequently proved in \cite{geometric}. Later, \citet{convex} proved the conjecture for $k=5$ and every $\ell\ge3$. In fact, they established the stronger statement that, for every fixed $\ell\ge3$, every sufficiently large finite planar point set contains either $\ell$ collinear points or an empty pentagon. Here, an \emph{empty pentagon} is a set of five
points in strictly convex position whose convex hull contains
no other point of $P$, either in its interior or on its boundary. Such a set of five vertices is therefore pairwise visible in $P$. Later, \citet{pentagonsharp} obtained the optimal order of growth in $\ell$, showing that the minimum number of points needed to guarantee $\ell$ collinear points or an empty pentagon is $\Theta(\ell^2)$. Until recently, the big-line-big-clique conjecture, which also appears as Problem 70 in Ben Green's \emph{100 open problems}~\cite{list}, remained open for every pair $(k,\ell)$, with $k\ge6$ and $\ell\ge4$. Very recently, \citet{Bonnet} made a major breakthrough by proving that every sufficiently large finite planar point set contains either four collinear points or six pairwise visible points, thereby resolving the first previously open case: $(k,\ell)=(6,4)$.

In this paper, we extend Bonnet's results~\cite{Bonnet} in several directions. To describe our results, recall that Bonnet's main theorem~\cite[Theorem~1]{Bonnet} guarantees a copy of $K_6$ in the visibility graph of every sufficiently large planar point set with no four collinear points. This led us to the following broader question: for which fixed graphs $H$ does the visibility graph of every sufficiently large planar point set with no four collinear points contain a copy of $H$? Our first result shows that this holds for every fixed graph with chromatic number at most five. 
Throughout, $\chi(H)$ denotes the chromatic number of $H$, and copies of graphs are understood to be ordinary, not necessarily induced. Further, a graph is called $H$-free if it contains no such copy of $H$.

\begin{theorem}\label{thm:five}
Let $H$ be a fixed graph with $\chi(H)\le5$. There is an integer $N_H$
such that every finite point set $P\subset\R^2$ with $|P|\ge N_H$
contains four collinear points or a copy of $H$ in $\V_P$.
\end{theorem}

The proof of Theorem~\ref{thm:five} is given in Section~\ref{sec:fivepf}. A high-level outline of the proof is given in Section \ref{sec:pf}.

\begin{remark} 
{\em Recall that the result of \citet{convex}, and the subsequent improvement by \citet{pentagonsharp}, guarantees an empty pentagon in every sufficiently large planar point set with no four collinear points, a conclusion that is substantially stronger than
the existence of a $5$-clique in its visibility graph. On the other hand, Theorem~\ref{thm:five} shows that, for every
fixed graph $H$ with $\chi(H)\le5$, the visibility graph of every
sufficiently large planar point set with no four collinear points
contains a copy of $H$. Note that any such graph $H$ can be embedded as a subgraph of the complete balanced 5-partite graph with $|V(H)|$ vertices in each part. Hence, Theorem~\ref{thm:five} is equivalent to the assertion that the visibility graph contains any fixed complete balanced 5-partite graph, if the point set is sufficiently large. Geometrically, this means that one can find five pairwise disjoint groups of points of any fixed sizes such that every point in one group is visible to every point in each of the other groups. No restriction is imposed on visibility within each group.  }  
\end{remark}

Next, we turn to graphs with chromatic number six. To this end, let $H=(V(H),E(H))$ be a fixed graph with $\chi(H)=6$. An edge $e\in E(H)$ is called {\it color-critical} if $\chi(H\backslash e)=\chi(H)-1=5$, where $H\backslash e$ is the graph obtained by deleting the edge $e$ from $H$. Color-critical edges play a classical role in extremal graph theory. A well-known result of \citet{Simonovits,symmetricgraph} shows that, for a fixed graph $H$ with $\chi(H)=r$, the complete balanced $(r-1)$-partite graph on $N$ vertices is the unique extremal $H$-free graph for all sufficiently large $N$ if and only if $H$ has a color-critical edge (see also \citet{RobertsScott} and \citet{Gerbner} for related results). Our next theorem shows that the visibility graph of every sufficiently large planar point set with no four collinear points contains a copy of every fixed 6-chromatic graph having a color-critical edge.

\begin{theorem}\label{thm:six}
Let $H$ be a fixed graph with $\chi(H)=6$ and a color-critical edge.
There is an integer $N_H$ such that every finite point set
$P\subset\R^2$ with $|P|\ge N_H$ contains four collinear points or a copy of $H$ in $\V_P$.
\end{theorem}

The proof of Theorem~\ref{thm:six} is given in Section~\ref{sec:sixpf}. A high-level outline of the proof is given in Section \ref{sec:pf}.  Going beyond $K_6$, another natural example of a graph with chromatic number six and a color-critical edge is $K_7\setminus\{e\}$, the graph obtained by deleting one edge from the 7-clique $K_7$ (see Figure \ref{fig:H} (a)). Hence, the following corollary is  an immediate consequence of Theorem~\ref{thm:six}.

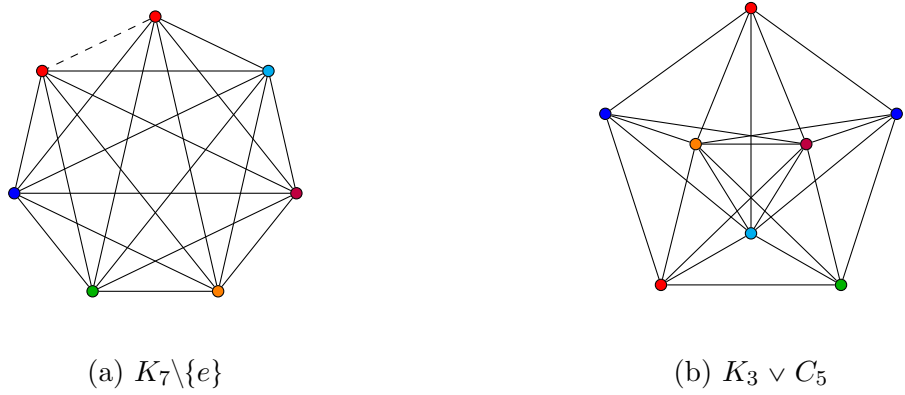
\begin{figure}[ht]
\centering
\begin{tikzpicture}[
  scale=1.125,
  vertex/.style={circle, draw=black, inner sep=0pt, minimum size=1.5mm},
  captionnode/.style={draw=none, fill=none}
]

\begin{scope}[xshift=0cm]

  \node[vertex, fill=red]            (a1) at ( 90:1.7) {};
  \node[vertex, fill=red]            (a2) at ({90+360/7}:1.7) {};
  \node[vertex, fill=blue]           (a3) at ({90+2*360/7}:1.7) {};
  \node[vertex, fill=green!70!black] (a4) at ({90+3*360/7}:1.7) {};
  \node[vertex, fill=orange]         (a5) at ({90+4*360/7}:1.7) {};
  \node[vertex, fill=purple]         (a6) at ({90+5*360/7}:1.7) {};
  \node[vertex, fill=cyan]           (a7) at ({90+6*360/7}:1.7) {};

  \draw (a1)--(a3); \draw (a1)--(a4); \draw (a1)--(a5); \draw (a1)--(a6); \draw (a1)--(a7);
  \draw (a2)--(a3); \draw (a2)--(a4); \draw (a2)--(a5); \draw (a2)--(a6); \draw (a2)--(a7);
  \draw (a3)--(a4); \draw (a3)--(a5); \draw (a3)--(a6); \draw (a3)--(a7);
  \draw (a4)--(a5); \draw (a4)--(a6); \draw (a4)--(a7);
  \draw (a5)--(a6); \draw (a5)--(a7);
  \draw (a6)--(a7);

  \draw[dashed] (a1)--(a2);

  \node[captionnode] at (0,-2.5) {(a) $K_7\setminus\{e\}$};

\end{scope}

\begin{scope}[xshift=7cm]

  \node[vertex, fill=red]            (b1) at ( 90:1.8) {};
  \node[vertex, fill=blue]           (b2) at ( 18:1.8) {};
  \node[vertex, fill=green!70!black] (b3) at (-54:1.8) {};
  \node[vertex, fill=red]            (b4) at (-126:1.8) {};
  \node[vertex, fill=blue]           (b5) at (162:1.8) {};

  \draw (b1)--(b2)--(b3)--(b4)--(b5)--(b1);

  \node[vertex, fill=orange] (c1) at (-0.65, 0.20) {};
  \node[vertex, fill=purple] (c2) at ( 0.65, 0.20) {};
  \node[vertex, fill=cyan]   (c3) at ( 0.00,-0.85) {};

  \draw (c1)--(c2)--(c3)--(c1);

  \foreach \u in {c1,c2,c3}{
    \foreach \v in {b1,b2,b3,b4,b5}{
      \draw (\u)--(\v);
    }
  }

  \node[captionnode] at (0,-2.5) {(b) $K_3 \vee C_5$};

\end{scope}

\end{tikzpicture}
\caption{ Graphs with chromatic number 6 and a color-critical edge. The vertex colors exhibit proper six-colorings. In (a), the dashed segment indicates the deleted edge. }
\label{fig:H}
\end{figure}

\begin{corollary}\label{cor:K}
Every sufficiently large finite point set $P\subset\R^2$ contains either four collinear points or a copy of $K_7\setminus\{e\}$ in $\V_P$. In particular, if no four
points of $P$ are collinear, then there are seven points of $P$ with at most one non-visible pair. 
\end{corollary}

\begin{remark}
{\em This result comes within a single visibility edge of the next open case, $(k,\ell)=(7,4)$, of the big-line-big-clique conjecture. Specifically, it shows that every sufficiently large planar point set $P$ with no four collinear points contains seven
points that are `almost' pairwise visible, that is, at least $20$ of the $\binom{7}{2}=21$ pairs are visible in $P$.  }  
\end{remark}

We now provide examples of different classes of 6-chromatic graphs containing a color-critical edge to illustrate the scope of Theorem \ref{thm:six}.

\begin{example}[Complete 6-partite graphs]\label{example:graph}
{\em A complete 6-partite graph has a color-critical edge if and only if at least two of its parts are singletons. More precisely, its color-critical edges are exactly those joining two singleton parts. Indeed, deleting such an edge merges the two singleton parts into one independent part, producing a complete 5-partite graph. Conversely, if
at least one endpoint of a deleted edge belongs to a part of size at least two, choosing a different representative from that part yields a $K_6$ that survives the deletion. Hence, the chromatic number remains six. Consequently, up to permuting the parts, the complete
6-partite graphs to which Theorem~\ref{thm:six} applies are
precisely $K_{1,1,a,b,c,d}$, where $a,b,c,d\ge1$. The simplest example is obtained by taking $a=b=c=d=1$, which gives $K_6$. The next natural case is $a=b=c=1$ and $d=2$, which gives the graph $K_7 \backslash \{e\}$ that is considered in Corollary \ref{cor:K}. Another important subfamily consists of the \emph{complete split graphs}
$$K_{1,1,1,1,1,t}=K_5\vee\overline{K_t},$$
for $t\ge1$.\footnote{For two graphs $G = (V(G), E(G))$ and $H= (V(H), E(H))$ on disjoint vertex sets, the \emph{join} of $G$ and $H$, denoted by $G\vee H$, is the graph with vertex set $V(G)\cup V(H)$ and edge set $E(G)\cup E(H) \cup \{ \{u,v\}: u\in V(G),\ v\in V(H)\}$. Hence, $G\vee H$ is obtained from the disjoint union of $G$ and $H$ by adding every edge between their vertex sets. } Such a graph consists of a $5$-clique and an independent set of size $t$, with all edges between the two sets present. Every edge within
the $5$-clique is color-critical. The cases $t=1$ and $t=2$ are $K_6$ and $K_7$ with one edge deleted, respectively. Thus, Theorem~\ref{thm:six} guarantees that, for every fixed $t$, every sufficiently large planar point set with no four collinear points contains five pairwise visible points having at least $t$ common neighbors in the ambient visibility graph. In contrast, $K_{1,2,2,2,2,2}$ and $K_{2,2,2,2,2,2}$
have no color-critical edge and therefore lie outside the
scope of Theorem~\ref{thm:six}. The corresponding containment
problem remains open for each of these graphs (see Question \ref{question} below).  }
\end{example}

\begin{example}[Generalized wheel]
{\em Another example of a 6-chromatic graph with a color-critical edge is shown in Figure \ref{fig:H} (b). This is the graph $H = K_3\join C_5$, the join of a triangle with a 5-cycle. Consequently, $\chi(H)=\chi(K_3)+\chi(C_5)=3+3=6$. Further, deleting an edge of the triangle lowers the chromatic number to $5$. This can be interpreted as a generalized wheel, with a 5-cycle forming the rim, together with three mutually adjacent hub vertices, each joined to every vertex of the rim. } 
\end{example}

\begin{example}[Beyond complete 6-partite examples]
{\em Let $H$ be obtained from the complete 5-partite graph $K_{3,2,2,2,2}$ by adding an edge $e=\{u, v\}$ inside its three-vertex part. Then $\chi(H)=6$, since
$H$ contains a $K_6$ and admits a proper 6-coloring. Moreover, $H\backslash \{e\}=K_{3,2,2,2,2}$ has chromatic number five, hence $e$ is a color-critical edge of $H$. Also, observe that $H$ is not a subgraph of any graph $K_{1,1,a,b,c,d}$, with $a,b,c,d\ge1$, which is the family considered in Example~\ref{example:graph}. Indeed, $H$ contains
vertex-disjoint copies of $K_6$ and $K_5$, whereas deleting the vertices of any $K_6$ from $K_{1,1,a,b,c,d}$ leaves a four-partite graph, which contains no $K_5$. 
This shows that the scope of Theorem~\ref{thm:six} is not just limited to subgraphs of the complete 6-partite family in Example \ref{example:graph}.  }  
\end{example}

\begin{remark}
{\em Every 6-chromatic graph on at most ten vertices has a color-critical edge. This is because, any proper 6-coloring of such a graph has at least two singleton color classes. Also, their vertices must be adjacent, and deleting their joining edge allows the two classes to be merged, reducing the chromatic number to five. This observation is also sharp, since $K_{1,2,2,2,2,2}$ has eleven vertices and chromatic number six, but deleting any edge leaves a $K_6$. Hence, eleven is the smallest possible order of a 6-chromatic graph with no color-critical edge. In particular, Theorem~\ref{thm:six} applies to all 6-chromatic graphs on at most ten vertices.  }  
\end{remark}

A question that remains is whether the color-critical-edge assumption can be removed from Theorem \ref{thm:six}.

\begin{question}\label{question}
Does the visibility graph of every sufficiently large planar point set with no four collinear points contain a copy of every fixed graph $H$ with $\chi(H)=6$?
\end{question}

\subsection{Overview of the Proofs and New Technical Ingredients}
\label{sec:pf}

To prove Theorem~\ref{thm:five}, we suppose, for a contradiction, that $P$ is a sufficiently large planar point set with no four collinear points and that $\V_P$ is $H$-free, where $H$ is a fixed graph with $\chi(H)\le5$. Following \citet{Bonnet}, we order $P$ by a generic linear projection, so that every consecutive interval retains its ambient visibility graph (Observation~\ref{obs:J}). The Erd\H{o}s--Stone--Simonovits theorem \cite{ErdosSimonovits,ErdosStone} then provides a constant $0<\rho< \frac{2}{5}$ such that every sufficiently long interval (in terms of the projection order) $J$ of $P$ contains at most $\rho|J|^2$ visible pairs. Then a counting argument shows that a point set with these properties must have bounded size (Lemma~\ref{lem:density}), which leads to a contradiction. The main ingredient in this counting argument is a new sliding-interval comparison between visible and non-visible pairs (Lemma~\ref{lem:interval}). It shows that the total number of non-visible pairs, counted with multiplicity over all sliding intervals of length $2K$, is at most the total number of visible pairs, counted with multiplicity over all sliding intervals of length $K$.

To prove Theorem~\ref{thm:six}, we begin by assuming, for a contradiction, that $P$ is a sufficiently large planar point set with no four collinear points and that $\V_P$ is $H$-free, where $H$ is a fixed graph with $\chi(H)=6$ and a color-critical edge. In this case, the Erd\H{o}s--Stone--Simonovits theorem gives an asymptotic upper bound of $(\frac{2}{5}+o(1))N^2$ on the number of edges in an $N$-vertex $H$-free graph, so a further argument is needed to obtain a fixed deficit below the $\frac{2}{5}$ threshold. 
Towards this, following \citet{Bonnet}, we first invoke a result of \citet{HujterKisfaludiBak}, which shows that planar point sets with no four collinear points and $5$-colorable visibility graphs have bounded size. Bonnet's interval-coloring argument \cite[Lemma 1]{Bonnet} then shows that the induced visibility graph on every sufficiently long interval requires the deletion of a fixed positive
proportion of its vertices to become $5$-colorable (see Lemma~\ref{lem:Z}). The new external input is a vertex-removal stability result for $H$-free graphs when $H$ has a color-critical edge~\cite{Gerbner,RobertsScott}. Combined with this deletion bound, we obtain a constant $\beta>0$, depending only on $H$, such that every sufficiently long
interval $J$ of $P$ contains at most $\left(\frac{2}{5} - \beta\right)|J|^2$ visible pairs.
The same sliding-interval comparison used in the proof of Theorem~\ref{thm:five}, then completes the argument.

The sliding-interval comparison replaces the harmonic-potential argument in Bonnet's proof~\cite{Bonnet}. This replacement also simplifies his original proof of the existence of a visible $K_6$ and improves the dependence of the resulting bound on the interval-density parameters. We illustrate this quantitative gain in Appendix~\ref{sec:count} by improving the explicit bound in \cite[Theorem~1]{Bonnet} on the number of points sufficient to guarantee four collinear points or six pairwise visible points.

\subsection{AI Declaration }

This paper grew out of an AI-assisted effort to further explore the consequences of the ideas developed by \citet{Bonnet}. The authors proposed investigating the existence of almost cliques and, more generally, copies of fixed graphs with chromatic number $5$ or $6$ in visibility graphs. Building on these ideas, GPT-6 Astra produced the initial proofs of Theorems~\ref{thm:five} and~\ref{thm:six} through a series of iterative exchanges with the authors. The authors subsequently revised the arguments and supplied additional details, examples, and contextual discussion. The authors assume responsibility for all content.

\section{Preliminaries} 
\label{sec:interval}

Throughout, $P\subset\R^2$ will be a finite set of $n$ points with no four collinear. Fix a non-vertical line $L$ such that the orthogonal projection $\pi$ of $P$ onto $L$ gives $|P|$ distinct points. We enumerate the points in $P$ as $p_1, p_2, \ldots, p_n$ by increasing order of their projections on $L$, that is,  $\pi(p_1) < \pi(p_2) < \cdots < \pi(p_n)$. Any set of consecutive points $\{p_u , p_{u+1}, \ldots, p_v\}$, with $1 \leq u \leq v \leq n$, will be referred to as an {\it interval} of $P$. Also, a point $r \in P$ will be called  a {\it blocker} for the points $p, q \in P$, if $r$ lies on the open line segment joining $p$ and $q$. 
Further, for any graph $H = (V(H), E(H))$ and $S \subseteq V(H)$, we denote by $H[S]$ the induced subgraph of $H$ with vertex set $S$. We begin by recalling the following observation from \cite{Bonnet}.

\begin{observation}[{\cite[Observation~1]{Bonnet}}]
\label{obs:J}
For every interval $J\subseteq P$, $\V_J = \V_P[J]$.
\end{observation}

\begin{proof}
Clearly, if two points of $J$ are visible in $P$, they remain visible in $J$. Conversely, suppose that $p_u, p_ v\in J$ have a blocker $p_w \in P$. Since projection on $L$
preserves strict betweenness, the index $w$ must lie between $u$ and $v$ (after interchanging $u$ and $v$, if necessary). Hence, $p_w \in J$, and $p_u $ and $p_v$ are also non-visible in $J$.
\end{proof}

The previous observation shows that the visibility graph of any interval $J$ of $P$ is, in fact, the induced subgraph of the visibility graph of $P$ on the set $J$. The next observation shows that $\V_P$ has a unique-blocker structure, since no four points of $P$ are collinear.

\begin{observation}\label{obs:structure}
If $\{p_u, p_v\}$ is a non-edge of $\V_P$, with $1 \leq u < v \leq n$, then there is a unique $w \in \{1, 2, \ldots, n\}$, with $u< w < v$, such that $p_w$ blocks $p_u, p_v$. Further, $\{p_u, p_w\}$ and $\{p_w, p_v\}$ are edges of $\V_P$, and the line through $p_u$ and $p_v$ contains exactly these three points of $P$.
\end{observation}

\begin{proof}
Since $p_u$ and $p_v$ are non-visible and no four points of $P$ are collinear,  a unique blocker $p_w$, with $u< w < v$, exists, and the line joining $p_u$ and $p_v$ has no further point of $P$. Consequently, both pairs $\{p_u, p_w\}$ and $\{p_w, p_v\}$ are visible.
\end{proof}

\section{A Sliding Interval Comparison}

A key ingredient of the proofs is a sliding-interval comparison between visible and non-visible pairs. To this end, for $1\le r\le n$ and $1\le s\le n-r+1$, let
$$J_r(s)=\{p_s,p_{s+1},\ldots,p_{s+r-1}\}  $$
be the interval consisting of the $r$ consecutive points of $P$
starting at $p_s$. Define
\begin{equation}\label{eq:Q}
\begin{aligned}
\mathcal Q_r &=\sum_{s=1}^{n-r+1} |E(\V_P[J_r(s)])|  ,  \\
\overline{\mathcal Q}_r&=\sum_{s=1}^{n-r+1} \left(\binom{r}{2}-|E(\V_P[J_r(s)])|\right)  .  
\end{aligned}
\end{equation}
Note that $|E(\V_P[J_r(s)])|$ and $\binom{r}{2}-|E(\V_P[J_r(s)])|$ count the number of visible and non-visible unordered pairs with both endpoints in $J_r(s)$,
where visibility is taken with respect to the full set $P$,  respectively. Consequently, $\mathcal Q_r$ and $\overline{\mathcal Q}_r$ are the counts of visible and non-visible unordered pairs aggregated over all intervals of length $r$, respectively. Note that a pair (visible or non-visible) is counted once for every sliding interval of length $r$
that contains both of its endpoints.


\begin{lemma}\label{lem:interval}
For every positive integer $K$ with $2 K \le n$,
\begin{equation}\label{eq:interval}
\overline{\mathcal{Q}}_{2K}\le \mathcal{Q}_K.
\end{equation}
\end{lemma}

\begin{proof}
Note that an object counted in $\overline{\mathcal{Q}}_{2K}$ consists of a non-visible pair $\{p_u, p_v\}$, with $1 \leq u < v \leq n$, and an interval $J_{2K}(s)$ containing both endpoints. Let $p_w \in P$ be the unique blocker of $\{p_u, p_v\}$. By Observation \ref{obs:structure}, $1 \leq u < w < v $ and both subpairs $\{ p_u, p_w\}$ and $\{ p_w, p_v\}$ are visible. Now, split the interval $J_{2K}(s)$ into two consecutive (disjoint) intervals: $J_{2K}(s)=J_K(s) \cup J_K(s+K)$. Note that, if $w \le s+K-1$, then $p_u$ and $p_w$ are both in $J_K(s) $, and if $w \ge s+K$, then $p_w$ and $p_v$ are both in $J_K(s+K)$. To prove \eqref{eq:interval} we now construct an injective map from objects counted in $\overline{\mathcal{Q}}_{2K}$ to objects counted in $\mathcal{Q}_{K}$ as follows: 
\begin{equation}\label{eq:Kls}
\psi(\{p_u,p_v\},J_{2K}(s))=
\begin{cases}
(\{p_u,p_w\},J_K(s)),& w \le s+K-1,\\ 
(\{p_w,p_v\},J_K(s+K)),& w \ge s+K.
\end{cases}
\end{equation}
Note that every image is an occurrence counted by $\mathcal{Q}_K$. To show that $\psi$ is injective, suppose that
$$\psi(\{p_u,p_v\},J_{2K}(s)) = \psi(\{p_{u'},p_{v'}\},J_{2K}(s')),$$ and let $p_w$ and $p_{w'}$ be the unique blockers of $\{p_u,p_v\}$ and $\{p_{u'},p_{v'}\}$, respectively. The common image contains a visible pair whose supporting line contains
both triples $\{p_u,p_w,p_v\}$ and $\{p_{u'},p_{w'},p_{v'}\}$. Since no four points of $P$ are collinear, these triples must
coincide. Moreover, the projection order gives $u<w<v$ and $u'<w'<v'$, which means, 
$u=u'$, $w=w'$, and $v=v'$. Further, the two visible subpairs $\{p_u,p_w\}$ and $\{p_w,p_v\}$ are distinct. Thus equality of the image pairs implies that both
preimages are mapped using the same case in
\eqref{eq:Kls}. If both use the first case, equality
of the image intervals gives $J_K(s)=J_K(s')$. If both use the second case, it gives
$J_K(s+K)=J_K(s'+K)$. In either case, the starting indices of the equal sliding
intervals coincide, and hence $s=s'$. Therefore, the two preimages are identical, proving that $\psi$ is injective. From this, it follows that $\overline{\mathcal{Q}}_{2K}\le\mathcal{Q}_K$, completing the proof of Lemma \ref{lem:interval}.  
\end{proof}

\begin{remark}  
{\em The above lemma shows that the total number of non-visible pairs, counted with multiplicity over all sliding intervals of length $2K$, is at most the total number of visible pairs counted with multiplicity over all sliding intervals of length $K$. The proof uses the properties of the projection order and the fact that each non-visible pair belongs to a unique three-point line. If four collinear points are allowed, a visible subpair can lie on several blocked pairs on the same line, and the reconstruction above need not be unique.  }  
\end{remark}

\section{Proof of Theorem \ref{thm:five}} 
\label{sec:fivepf}

We begin by showing that a point set with no four collinear points cannot be arbitrarily large if every sufficiently long sliding interval $J$ contains at most $\rho|J|^2$ visible pairs, for some fixed $0 < \rho< \frac{2}{5} $.

\begin{lemma}\label{lem:density}
Suppose $0<\rho<\frac{2}{5}$ and $K_0$ is a positive integer such that $|E(\V_P[J])|\le\rho |J|^2$, for every interval $J$ with $|J|\ge K_0$. Let $K \ge K_0$ be an integer satisfying $K > \frac{1}{2-5\rho}$. Then 
\begin{equation}\label{eq:density}
n\le 2K-1+\frac{\rho K^2}{(2-5\rho)K-1}   .  
\end{equation}  
\end{lemma}

\begin{proof}
If $n \leq 2K-1$, \eqref{eq:density} trivially holds. Hence, we can assume $n \geq 2K$. Note that there are exactly $n-K+1$ intervals of length $K$. Therefore, from the hypothesis in Lemma \ref{lem:density} and recalling \eqref{eq:Q} gives, 
\begin{equation*}
\mathcal{Q}_K\le(n- K +1)\rho K^2.
\end{equation*}
Also, each of the $n- 2 K   +1$ intervals of length $2 K   $ contains at least $\binom{2 K   }{2}- 4 \rho K ^2=(2-4\rho)K^2- K$ non-visible pairs. Therefore, once again recalling \eqref{eq:Q}, 
\begin{equation*}
\overline{\mathcal{Q}}_{2K}\ge(n- 2 K   +1)((2-4\rho)K^2- K ).
\end{equation*}
Combining the above with Lemma~\ref{lem:interval} gives, 
$$(n- 2 K   +1)((2-4\rho)K^2- K )
\le(n- K +1)\rho K^2.$$
Using $n-K+1=(n-2K+1)+K$, the preceding inequality becomes
$$(n-2K+1)\bigl((2-4\rho)K^2-K\bigr)
\le (n-2K+1)\rho K^2+\rho K^3.$$
Subtracting $(n-2K+1)\rho K^2$ from both sides and dividing by $K$ gives, 
$$(n-2K+1)\bigl((2-5\rho)K-1\bigr)
\le\rho K^2.$$
The assumption $K> \frac{1}{2-5\rho}$ ensures that $(2-5\rho)K-1>0$. Therefore, dividing by $(2-5\rho)K-1$ and rearranging terms, the result in Lemma \ref{lem:density} follows. 
\end{proof}

\begin{remark}\label{remark:interval}
{\em The parameter $\frac{2}{5}$ is the exact threshold for the counting in the above lemma: the leading terms compare $2-4\rho$ with $\rho$. The strict inequality $\rho<\frac{2}{5}$ is what turns the comparison into an upper bound on $n$.  }  
\end{remark}

Now, let $H = (V(H), E(H))$ be a fixed graph with chromatic number $1 \leq \chi(H) \leq 5$. Suppose, for a contradiction, that $\V_P$ is $H$-free. For an integer $a \geq 1$, denote by $K_5(a)$ the complete balanced $5$-partite graph with
$a$ vertices in each part. Choose a proper coloring of $H$ with at most five
colors, allowing empty color classes. Then by mapping each color class injectively
into one part of $K_5(|V(H)|)$ shows that $H\subseteq K_5(|V(H)|)$. Hence, every $H$-free graph is $K_5(|V(H)|)$-free. In particular, $\V_P$ is $K_5(|V(H)|)$-free. The following result is the classical Erd\H{o}s--Stone--Simonovits Theorem \cite{ErdosSimonovits,ErdosStone}, specialized to the case of the complete 5-partite graph.

\begin{theorem}[{\cite{ErdosSimonovits,ErdosStone}}] 
\label{thm:partite}
Fix $\varepsilon>0$ and an integer $a \geq 1$. Then, there exists $N(\varepsilon, a) \geq 1$ such that every $K_5(a)$-free graph $G = (V(G), E(G))$, with $|V(G)| \geq N(\varepsilon, a)$ satisfies
$$|E(G)| \le\left(\frac{3}{8}+\varepsilon\right)|V(G)|^2.$$
\end{theorem}

For the proof of Theorem~\ref{thm:five}, apply
Theorem~\ref{thm:partite} with $\varepsilon= \frac{1}{120}$ and
$a=|V(H)|$. This gives an integer $M_H\ge5|V(H)|$ such that
every $K_5(|V(H)|)$-free graph $G$ with $|V(G)|\ge M_H$
satisfies
$$|E(G)|\le\rho |V(G)|^2,$$
with $\rho=\frac{23}{60}=\frac38+\frac1{120}<\frac25$. Define $K_H=\max\{M_H,60\}$ and $N_H=8K_H$, and suppose that $n=|P|\ge N_H$. Since $\chi(H)\le5$, we have $H\subseteq K_5(|V(H)|)$. Thus, $\V_P$, and hence every induced interval graph $\V_P[J]$, is $K_5(|V(H)|)$-free. Consequently, $|E(\V_P[J])|\le\rho |J|^2$, for every interval $J$ with $|J|\ge M_H$. Since $n\ge8K_H$, sliding intervals of both lengths $K_H$ and $2K_H$ exist. Also, $K_H\ge M_H$ and $K_H\ge 60>12=\frac1{2-5\rho}$. Hence, applying Lemma~\ref{lem:density} with $K=K_H$ and $K_0=M_H$ we obtain, 
$$ n \le 2K_H-1+ \frac{\rho K_H^2}{(2-5\rho)K_H-1} \le 2K_H-1+  \frac{\frac{23}{60}K_H^2}{\frac{1}{15} K_H} = \frac{31}{4}K_H-1 < 8K_H=N_H  ,  $$
contradicting $n\ge N_H$.  \hfill $\Box$

\section{Proof of Theorem \ref{thm:six}}
\label{sec:sixpf}

For a fixed 6-chromatic graph $H$, the Erd\H{o}s--Stone--Simonovits theorem gives the leading extremal edge density coefficient of $\frac{2}{5}$, which does not provide the strict
deficit required by Lemma \ref{lem:density}. For this we need additional ingredients. The first is
the following bound on the size of planar point sets with no four collinear points and $5$-colorable visibility graphs.

\begin{theorem}[\citet{HujterKisfaludiBak}]\label{thm:45}
Every finite planar point set of size at least $2311$ contains four
collinear points or has a visibility graph with chromatic number at
least six.
\end{theorem}

Theorem \ref{thm:45} implies that a point set with no four collinear and a $5$-colorable
visibility graph has at most $n_{\mathsf{HKB}}=2310$ points (where the subscript stands for the authors). Now, define 
\begin{equation}\label{eq:nepsilon}
\varepsilon_0=\frac{1}{2(n_{\mathsf{HKB}}+1)}=\frac1{4622}.
\end{equation}
Now, for a graph $G = (V(G), E(G))$, define its vertex-deletion distance from
$5$-colorability by
\begin{align}\label{eq:distanceG} 
\tau_5(G)=\min\{|Z|: Z\subseteq V(G),\ \chi(G[V(G)\backslash Z])\le5\}.
\end{align}
A consequence of Theorem \ref{thm:45} is the following lemma from \cite{Bonnet}. We reproduce the proof for completeness.

\begin{lemma}[{\cite[Lemma 1]{Bonnet}}]
\label{lem:Z}
For any interval $J$ of $P$, with $|J|\ge 2 n_{\mathsf{HKB}}$, $\tau_5(\V_P[J])\ge\varepsilon_0 |J|$. 
\end{lemma}

\begin{proof}
Let $Z\subseteq J$ be such that $\chi(\V_P[J \backslash Z])\le5$. The set $J \setminus Z$ is the union of $a \leq |Z|+1$ nonempty consecutive intervals $J_1,\ldots,J_{a}$ in the original projection order. For every $1 \leq b \leq a$, Observation~\ref{obs:J} gives $\V_{J_b}=\V_P[J_b]$. Since $\V_P[J_b]$ is an induced subgraph of $\V_P[J \backslash Z]$, it is $5$-colorable. Then, Theorem~\ref{thm:45} implies that $|J_b|\le n_{\mathsf{HKB}}$. Hence, 
$$|J|-|Z|=\sum_{b=1}^{a}|J_b|\le n_{\mathsf{HKB}} (|Z|+1).$$
Then using $|J| \geq 2 n_{\mathsf{HKB}}$ gives, $|Z| \ge\frac{|J|-n_{\mathsf{HKB}}}{n_{\mathsf{HKB}} +1}\ge\frac{|J|}{2(n_{\mathsf{HKB}} + 1)}=\varepsilon_0 |J|$. This proves Lemma \ref{lem:Z}.  
\end{proof}

The main technical input for the proof of Theorem \ref{thm:six} is a vertex-removal stability result for $H$-free graphs, where $H$ has chromatic number six and a color-critical edge. It relates their edge counts to their vertex-deletion distance from $5$-colorability. This is a qualitative consequence of \cite[Lemma 1.5]{Gerbner} (see also the proof of Lemma 2.3 in \cite{RobertsScott}). We include a proof of this result in Appendix \ref{sec:coloredgepf} for the sake of completeness.

\begin{proposition}\label{ppn:coloredge}
Let $H$ be a fixed graph with $\chi(H)=6$ and a color-critical
edge. For every $\delta>0$, there exist a constant $\beta=\beta(H,\delta)>0$ and a positive integer $M_H=M_H(\delta)$ such that every $H$-free graph $G = (V(G), E(G))$, with $|V(G)| \ge M_H$ and $\tau_5(G)\ge\delta |V(G)|$, satisfies
\begin{equation}\label{eq:VG}
|E(G)|\le\left(\frac25-\beta\right) |V(G)|^2.
\end{equation}
\end{proposition}

\begin{remark}
{\em Recall that the number of edges in a balanced complete 5-partite graph with $N$ vertices in total is at most $\frac{2}{5} N^2$, with equality when $5$ divides $N$. Proposition~\ref{ppn:coloredge} shows that a sufficiently large $H$-free graph on $N$ vertices that requires the deletion of at least $\delta N$ vertices to become $5$-colorable has at least $\beta N^2$ fewer edges than $\frac{2}{5}N^2$. Thus, being far from $5$-colorable forces a fixed quadratic edge deficit below the leading 5-partite Tur\'an bound.  }  
\end{remark}

Using Proposition \ref{ppn:coloredge} we can show in the next lemma that forbidding a fixed graph $H$ as above, in the visibility graph of a point set with no four collinear points forces every sufficiently long sliding interval $J$ to have a fixed quadratic edge deficit below $\frac{2}{5}|J|^2$.

\begin{lemma}
\label{lem:intervaledge}
Let $H$ be a fixed graph with $\chi(H)=6$ and a color-critical edge. There exist constants $0<\beta\le\frac{1}{10}$ and an integer $K_0$ such that, whenever $P$ has no four collinear points and $\V_P$ is $H$-free, then every interval $J$ of $P$, with $|J| \ge K_0$, satisfies 
\begin{equation}\label{eq:interval-deficit}
|E(\V_P[J])|\le\left(\frac{2}{5}-\beta\right) |J|^2  .  
\end{equation}
\end{lemma}

\begin{proof}
Apply Proposition \ref{ppn:coloredge} with $\delta=\varepsilon_0$, to obtain $\beta>0$ and $M_H$. Note that decreasing $\beta$, if necessary, still preserves the result in Proposition~\ref{ppn:coloredge}. Hence, we can assume that $0<\beta\le\frac{1}{10}$. Define $$K_0=\max \{2 n_{\mathrm{HKB}}, M_H \}.$$ For an interval $J$ of size $|J| \ge K_0$, the induced graph $\V_P[J]$ is
$H$-free, and Lemma~\ref{lem:Z} gives
$\tau_5(\V_P[J])\ge\varepsilon_0 |J|$. Then \eqref{eq:VG} implies, 
$$|E(\V_P[J])|\le \left(\frac{2}{5} - \beta\right) |J|^2,$$
as required.
\end{proof}

Fix a graph $H$ with $\chi(H)=6$ and a color-critical edge.
Choose $\beta$ and $K_0$ as in Lemma~\ref{lem:intervaledge}, and set
$$\rho:=\frac25-\beta \quad \text{ and } \quad m:=\max\left\{K_0,\left\lceil\frac{2}{5\beta}\right\rceil\right\}.$$
Define
\begin{align*}
N_H:=1+\left\lceil \left(2+\frac{2\rho}{5\beta}\right) m \right\rceil.
\end{align*} 
 Suppose, for a contradiction, that $P$ is a point set with $n:=|P|\ge N_H$, no four points of $P$ are collinear, and $\V_P$ is $H$-free. By Lemma~\ref{lem:intervaledge}, every interval $J$ of $P$ with $|J|\ge K_0$ satisfies $|E(\V_P[J])|\le\rho |J|^2$. Since $0<\beta\le \frac{1}{10}$, we have $0<\rho< \frac{2}{5}$. Further, $m \ge\frac{2}{5\beta}> \frac{1}{2-5\rho}$. Thus, Lemma~\ref{lem:density} gives, 
$$n\le 2 m -1+ \frac{\rho m ^2}{(2-5\rho) m -1} \le 2 m -1+ \frac{\rho m^2}{ \frac{5}{2} \beta m } <N_H, $$ 
contradicting $n\ge N_H$. This proves the theorem.  \hfill  $\Box$

%
%

\begin{remark} 
{\em The present argument appears to reach a natural barrier at $K_7$. The sliding-interval comparison becomes effective when every sufficiently long interval $J$ contains at most $\rho|J|^2$ visible pairs for some fixed $\rho<\frac{2}{5}$, with
$\frac{2}{5}$ being the exact threshold for this counting argument (recall Remark~\ref{remark:interval}). Interestingly, the Erd\H{o}s--Stone--Simonovits theorem shows that the maximum number of edges in an $N$-vertex graph avoiding a fixed
6-chromatic graph is $(\frac{2}{5}+o(1))N^2$. For 6-chromatic graphs with a color-critical edge, stability, combined with the interval-coloring argument, yields a fixed positive deficit below this threshold for the relevant visibility graphs (as in Lemma~\ref{lem:intervaledge}). In contrast, the graph-theoretic condition of being $K_7$-free alone gives only the Tur\'an bound: the maximum number of edges in a $K_7$-free graph on $N$ vertices is $\frac{5}{12}N^2+O(1)$, which lies above the $\frac{2}{5}$ threshold, and the stability input used here is not enough to bridge the gap in the leading coefficients. Thus, resolving the $(7,4)$-case of the big-line-big-clique conjecture would likely require additional geometric inputs or a stronger comparison than the present
sliding-interval argument.  }  
\end{remark}

\small  

\subsection*{Acknowledgments}  
BBB is grateful to Arijit Bishnu for introducing him to the big-line-big-clique conjecture nearly fifteen years ago, and for many interesting discussions over the years. BBB was supported in part by NSF CAREER Grant DMS-2046393 and the National University of Singapore's PESS fund.

\bibliographystyle{abbrvnat} 
\bibliography{bibliography}

\normalsize

\appendix

\section{Improved Bound for Six Visible Points}\label{sec:count}

In this section, we show how the sliding-interval counting argument, combined with the following explicit $K_6$-stability result from \cite{Bonnet}, improves the bound on the number of points required to guarantee a 6-clique in a visibility graph with no four collinear points.


\begin{lemma}[{\cite[Lemma~2]{Bonnet}}]\label{lem:clique}
Let $0<\varepsilon< \frac{1}{3750}$. If a $K_6$-free graph $G = (V(G), E(G))$, with $|V(G)|\ge20 $ vertices satisfies
$$
|E(G)|> \left( \frac{2}{5}-\varepsilon \right) |V(G)|^2,
$$
then there is a set $Z\subseteq V(G)$ with $|Z|<3500\varepsilon |V(G)|$  and $
\chi(G[V(G)\backslash Z])\le 5$. 
\end{lemma}

Combining this result with Lemma~\ref{lem:density} yields the following numerical improvement on Bonnet's bound~\cite{Bonnet}.

\begin{corollary}\label{cor:six}
Every finite planar point set with at least $2\cdot10^{13}$ points
contains four collinear points or six pairwise visible points.
\end{corollary}

\begin{proof}
Suppose, for a contradiction, that $P\subset\R^2$ is a finite
point set with $n:=|P|\ge2\cdot10^{13}$, no four points of $P$
are collinear, and $\V_P$ is $K_6$-free. We order $P$ by a generic
linear projection as in Section \ref{sec:interval}, and set
$$\varepsilon_0:=\frac1{4622}, \quad \beta:=\frac{\varepsilon_0}{3500}
=\frac1{16\,177\,000}, \quad K_0:=4620, \quad \text{ and } \quad  \rho:=\frac25-\beta.$$
We first show that every interval $J$ of $P$, with $|J|\ge K_0$, 
satisfies $|E(\V_P[J])|\le\rho|J|^2$. To see this, first note that by Lemma~\ref{lem:Z}, 
$\tau_5(\V_P[J])\ge\varepsilon_0|J|$. Hence, if $|E(\V_P[J])|> (\frac25-\beta)|J|^2$,
then Lemma~\ref{lem:clique}, applied to the $K_6$-free graph
$\V_P[J]$ with $\varepsilon=\beta$, would give a set
$Z\subseteq J$ satisfying
$$
|Z|<3500\beta|J|=\varepsilon_0|J|
\qquad\text{and}\qquad
\chi(\V_P[J\setminus Z])\le5,
$$
which is a contradiction to $\tau_5(\V_P[J])\ge\varepsilon_0|J|$. Note that the hypotheses of Lemma~\ref{lem:clique} hold because $|J|\ge4620\ge20$ and $0<\beta< \frac{1}{3750}$.
 
Now, choose
$$m:=\max\left\{ K_0,\left\lceil\frac{2}{5\beta}\right\rceil
\right\} =6\,470\,800.$$
Note that $m =\frac{2}{5\beta}> \frac{1}{2-5\rho}$. Hence, by Lemma~\ref{lem:density}, 
\begin{align}\label{eq:clique}
n\le2 m-1+ \frac{\rho m^2}{(2-5\rho)m-1} =\left(2+\frac{2\rho}{5\beta}\right)m-1 =16\,748\,511\,409\,279 <2\cdot10^{13}, 
\end{align}
contradicting $n\ge2\cdot10^{13}$. 
\end{proof}

%
%

\begin{remark}  
{\em Corollary~\ref{cor:six} shows that $2\cdot10^{13}$ points suffice to guarantee four collinear points or six pairwise visible points, improving on the bound $10^{11055931}$ in \cite[Theorem~1]{Bonnet}. This improvement retains Bonnet's original stability and coloring inputs and illustrates the substantial quantitative gain obtained from the sliding-interval counting argument. A further improvement is possible by combining the bound in \cite[Theorem~15]{HujterKisfaludiBak} with the exact value of the empty-hexagon number. Let $\mathsf{H}(6)$ denote the smallest integer such that every finite planar point set with at least $\mathsf{H}(6)$ points in general position (that is, with no three collinear points) contains an empty convex hexagon. \citet[Theorem~15]{HujterKisfaludiBak} show that a planar point set with no four collinear points and a $5$-colorable visibility graph has at most $5\mathsf{H}(6)-5$ points. Since \citet{HeuleScheucher} proved that $\mathsf{H}(6)=30$, we may replace $n_{\mathrm{HKB}}=2310$ by $n_{\mathrm{HKB}}=145$. Keeping the stability constant $3500$ unchanged and repeating the proof of Corollary~\ref{cor:six} with the updated parameters improves the bound in \eqref{eq:clique} to 
$$n\le66\,847\,630\,079<7\cdot10^{10}.$$ This shows, every finite planar point set with at least $7\cdot10^{10}$ points contains four collinear points or six pairwise visible points. During the final stages of preparing this manuscript, we became aware of a new arXiv preprint \cite{improved} which establishes that $880$ points suffice for the same conclusion, a substantial improvement over the numerical bounds obtained here. Determining the exact minimum number of points required remains open.  }  
\end{remark}

\section{ Proof of Proposition \ref{ppn:coloredge} }
\label{sec:coloredgepf}

We use the following partition formulation of the
Erd\H{o}s-Simonovits stability theorem
\cite{ErdosStability,Simonovits} (see also 
\cite[Theorem 1.1]{RobertsScott}).

\begin{theorem}  
\label{thm:stability}
Let $F$ be a fixed graph with $\chi(F)=6$, and let $\{G_t\}_{t \geq 1}$, with $G_t = (V(G_t), E(G_t))$, be a sequence of $F$-free graphs with $|V(G_t)| \to\infty$ and
$$|E(G_t)|\ge\left(\frac25-o(1)\right)|V(G_t)|^2.$$
Then there are partitions $V(G_t)=V_{t}^{(1)} \cup\cdots\cup V_{t}^{(5)}$, such that
$$|V_{t}^{(i)}|=\frac{|V(G_t)|}{5}+o(|V(G_t)|) , $$
for every $1 \leq i \leq 5$, and the total number of nonadjacent pairs with endpoints
in different parts is $o(|V(G_t)|^2)$. Further, the total number of edges with both endpoints in the same part is $o(|V(G_t)|^2)$.  
\end{theorem}

\begin{proof}[Proof of Proposition~\ref{ppn:coloredge}]
Fix a graph $H$ with $\chi(H)=6$ and a color-critical edge. Suppose, for a contradiction, that Proposition~\ref{ppn:coloredge} fails for some $\delta>0$. Then, for every positive integer $t$, there exists an $H$-free graph $G_t$ on $|V(G_t)|\ge t$ vertices such that
\begin{equation}\label{eq:stability-countersequence}
|E(G_t)|>
\left( \frac{2}{5} - \frac{1}{t} \right)|V(G_t)|^2
\qquad\text{and}\qquad
\tau_5(G_t)\ge\delta |V(G_t)|.
\end{equation}
In particular, $|V(G_t)|\to\infty$ and $|E(G_t)|\ge( \frac{2}{5}-o(1))|V(G_t)|^2$. Hence, we can apply Theorem~\ref{thm:stability} with $F=H$ to obtain
partitions $V(G_t)=V_{t}^{(1)} \cup\cdots \cup V_{t}^{(5)}$. Let $\mu_t=o(|V(G_t)|^2)$ be the number of missing cross-edges, that is, nonadjacent pairs with endpoints in different parts. Define 
$$\gamma_t:= \max\left\{ \frac{1}{\sqrt{|V(G_t)|}} , \frac{\sqrt{\mu_t}}{|V(G_t)|}
\right\},$$
and let $B_t$ consist of all vertices incident on at least $\gamma_t |V(G_t)|$ missing cross-edges. Counting incidences with these pairs gives $|B_t|\gamma_t |V(G_t)|\le2\mu_t$. Using this and $\gamma_t|V(G_t)|\ge\sqrt{\mu_t}$ gives, 
$$|B_t| \leq \frac{ 2 \mu_t}{ \gamma_t |V(G_t)|}  \le 2\sqrt{\mu_t}
=o(|V(G_t)|).$$ Hence, every vertex in $V_{t}^{(i)}\setminus B_t$ has fewer than
$\gamma_t |V(G_t)|$ nonneighbors outside $V_{t}^{(i)}$. Further, 
\begin{align}\label{eq:Vt}
|V_{t}^{(i)}\setminus B_t| =\frac{|V(G_t)|}{5}+o(|V(G_t)|), 
\end{align}
for every $1 \leq i \leq 5$.  

Now, suppose $e=\{u,v\}$ is a color-critical edge of $H$. Since $\chi(H\backslash\{e\})=5$, fix a proper 5-coloring of $H\backslash\{e\}$, with color classes $C_1,\ldots,C_5$. The vertices $u$ and $v$ must receive the same color, otherwise this would also be a proper 5-coloring of $H$, which contradicts $\chi(H)=6$. Relabeling the classes if necessary, assume that $u,v\in C_1$. Hence, $e$ is the only edge of $H$ whose endpoints belong to the same color class. Using the $H$-free property of $G_t$ we can show the following:

\begin{observation}\label{obs:VB}
For all sufficiently large $t$, each $V_{t}^{(i)}\setminus B_t$ is an independent set in $G_t$, for all $1 \leq i \leq 5$.  
\end{observation}

\begin{proof} 
After relabeling the parts, suppose for a contradiction, $V_{t}^{(1)}\setminus B_t$ is not an independent set in $G_t$. Fix an edge $\{x,y\}\in E(G_t[V_{t}^{(1)}\setminus B_t])$.  Recall that $C_1,\ldots,C_5$ are the color classes of a proper 5-coloring of $H\setminus\{e\}$, where $e=\{u,v\}$ and $u,v\in C_1$. We will construct an injective map $\sigma:V(H) \rightarrow V(G_t)\setminus B_t$, 
such that
$$ \sigma(u)=x,\qquad \sigma(v)=y, \quad \text{ and } \quad \sigma(C_i)\subseteq V_{t}^{(i)}\setminus B_t  , $$
for $1 \leq i \leq 5$, and the images of every two vertices in different color classes are adjacent in $G_t$. These properties will ensure that $\sigma$ embeds $H$
as a subgraph of $G_t$, which will contradict the assumption that $G_t$ is $H$-free.

To construct such a function $\sigma$, first map $u$ to $x$ and $v$ to $y$. Then map the remaining vertices of $C_1$ to arbitrary distinct unused vertices
of $V_{t}^{(1)}\setminus B_t$. No further adjacency requirements are needed
within $C_1$, since its only edge in $H$ is $e$. Next, embed the vertices of $C_2,\ldots,C_5$ one at a time. Suppose that the next vertex to be embedded belongs to $C_i$. We choose its image in $V_{t}^{(i)}\setminus B_t$, requiring it to
be unused and adjacent to every previously selected image lying in a different part. Each such image lies outside $B_t$ and therefore has fewer than $\gamma_t |V(G_t)|$ nonneighbors in $V_{t}^{(i)}\setminus B_t$. Since there are at most $|V(H)|$
previously selected images, the adjacency requirements exclude at most $\gamma_t |V(H)| |V(G_t)|$ candidates. At most $|V(H)|$ additional candidates have already been used. Hence, at least
$$|V_{t}^{(i)}\setminus B_t |- \gamma_t |V(H)| |V(G_t)|- |V(H)|>0$$
choices remain.  The strict positivity above follows since $|V_{t}^{(i)}\setminus B_t|=\frac{1}{5} |V(G_t)|+o(|V(G_t)|)$ (recall \eqref{eq:Vt}), $\gamma_t\to0$, and $|V(H)|$
is fixed. This shows that the greedy construction of the function $\sigma$ with the required properties  succeeds.  This leads to a contradiction, as explained before, and completes the proof of Observation \ref{obs:VB}.  
\end{proof}

By Observation \ref{obs:VB}, the five sets $V_t^{(1)}\setminus B_t,\ldots,V_t^{(5)}\setminus B_t$ form a proper five-coloring of $G_t[V(G_t)\setminus B_t]$. Consequently, $\tau_5(G_t)\le|B_t|=o(|V(G_t)|)$, contradicting $\tau_5(G_t)\ge\delta|V(G_t)|$ in \eqref{eq:stability-countersequence}. This completes the
proof of Proposition~\ref{ppn:coloredge}.  
\end{proof}

\end{document}